\documentclass{au}
\usepackage{amssymb,latexsym}
\usepackage{amsmath}
\usepackage[dvipdfm]{graphicx}

\newcommand{\bgq}{\boldsymbol{\theta}}

\newcommand{\gf}{\varphi}
\newcommand{\setm}[2]{\{#1:#2\}}
\newcommand{\setbm}[2]{\bigl\{#1:#2\bigr\}}
\newcommand{\set}[1]{\{#1\}}
\newcommand{\tbf}{\textbf}
\newcommand \alg[1]{{\mathcal #1}}
\newcommand \var[1]{\pmb{\mathcal #1}}
\newcommand \brick{{E}}
\newcommand \hideit[1]{}
\newcommand \minta {{D}}
\newcommand \mintc {{C}}

\newtheorem*{theoremnonum}{Theorem}
\def\sajat#1{}
 \usepackage{color}

\begin{document}
\title[Tolerances as images of congruences]{Tolerances as images of congruences in  varieties defined by linear identities}

\author[I.\ Chajda]{Ivan Chajda}
\email{ivan.chajda@upol.cz}
\address{Palack\'y University Olomouc\\Department of Algebra and Geometry\\17. listopadu 12,
771 46 Olomouc, Czech Republic}

\author[G.\ Cz\'edli]{G\'abor Cz\'edli}
\email{czedli@math.u-szeged.hu}
\urladdr{http://www.math.u-szeged.hu/~czedli/}
\address{University of Szeged\\ Bolyai Institute\\Szeged,
Aradi v\'ertan\'uk tere 1\\ Hungary 6720}

\author[R.\ Hala\v s]{Radom\'\i r Hala\v s}
\email{radomir.halas@upol.cz} 
\address{Palack\'y University Olomouc\\Department of Algebra and Geometry\\17. listopadu 12,
771 46 Olomouc, Czech Republic}

\author[P.\ Lipparini]{Paolo Lipparini}
\email{lipparin@mat.uniroma2.it}
\address{Department of Mathematics\\ Tor Vergata University of Rome\\ I-00133 Rome, Italy}

\thanks{This research was supported the project Algebraic Methods in Quantum Logic, No.:
CZ.1.07/2.3.00/20.0051 and by 
by the NFSR of Hungary (OTKA), grant numbers   K77432 and K83219}

\subjclass[2010]{Primary: 08A30. \sajat{Subalgebras, congruence relations} Secondary: 08B99\sajat{Varieties, none of the above}, 20M07\sajat{Varieties and pseudovarieties of semigroups}}

\keywords{Tolerance relation,  homomorphic image of a congruence, linear identity, balanced identity}

\date{June 5, 2012}

\begin{abstract} An identity $s=t$ is \emph{linear} if each variable occurs at most once in each of the terms $s$ and $t$. Let $T$ be a tolerance relation of an algebra $\alg A$ in a variety defined by a set of linear identities. We prove that there exist an algebra $\alg B$ in the same variety and a congruence $\bgq$ of $\alg B$ such that a homomorphism from $\alg B$ onto $\alg A$ maps $\bgq$ onto $T$.
\end{abstract}

%%*********************************************************************************************
%\red{Dear Referee,}
%
%\red{Thank you for the report. Your suggestion has been carried out, with a slight change. We would like to present the result as early in the paper as possible, as early as the necessary concepts are introduced. Hence the suggested sentences are in the proof rather than before the Theorem.  All the changes are in red.}
%\red{\hfill Yours sincerely: G\'abor Cz\'edli}
%
%
%
%%******************************************************************************************

Mailbox
\maketitle

An identity $s=t$ is \emph{linear} if each variable occurs at most once in each of the terms $s$ and $t$, see, for example,  M.\,N.\ Bleicher,  H.~Schneider and R.\,L.~Wilson~\cite[Theorem 4.19]{bleicher}, W.~Taylor~\cite{taylor},  I.~Bo\v snjak and R.~Madar\'asz~\cite{bosnjakmadarasz}, A.~Pilitowska~\cite{pilitowska}, and their references. 
In the particular case where every
variable occurs exactly twice,
once in $s$ and once in $t$, 
we speak of a \emph{balanced linear} identity,  see M.\,V.\ Lawson~\cite{lawson}. For example, the variety of semigroups and that of commutative semigroups are defined by balanced linear identities.   
Binary reflexive, symmetric, and compatible relations are called \emph{tolerances}; see I.~Chajda~\cite{chajdabook}.
%for a monograph devoted to them.
If $\gf\colon\alg B\to \alg A$ is a  surjective homomorphism and $\bgq$ is a congruence of the algebra $\alg B$, then $\gf(\bgq)=\set{(\gf(x),\gf(y)): (x,y)\in\bgq}$ is a  tolerance of $\alg A$. Each tolerance of $\alg A$ is obtained this way; this follows from our result below (applied for the variety defined by the empty set of linear identities). Sometimes, like in I.\ Chajda, G.\ Cz\'edli, and R.\ Hala\v s~\cite{chajdaczghalas} or  
G.\ Cz\'edli and G.~Gr\"atzer~\cite{czggg}, we can choose an appropriate $\alg B$ from a given variety. We have the following additional result of this kind.

\begin{theoremnonum}\label{aieosS}
Assume that $\var V$ is a variety defined by a set of linear identities, that $\alg A=(A,F)\in\var V$, and that $T$ is a tolerance of $\alg A$.
Then there exist an algebra $\alg B\in\var V$, a congruence $\bgq$ of $\alg B$, and a surjective homomorphism $\gf\colon\alg B\to \alg A$ such that 
$T=\gf(\bgq)$.
\end{theoremnonum}

\begin{proof} We generalize the idea of  G.~Cz\'edli and G.~Gr\"atzer~\cite{czggg}.

If $\alg \minta $ is an arbitrary algebra (not necessarily in $\var V$), then the \emph{complex algebra} $\alg \mintc$  of $\alg \minta $, in other words the \emph{algebra of complexes of $\alg \minta$}, has the underlying
set $\set{X\subseteq   \minta : X \neq \varnothing}$, and for each basic operation $f$ of $\alg \minta $, the corresponding
operation of $\alg \mintc$ is defined by 
\[
f(X_1,\ldots,X_n)=\setm{f(x_1,\ldots,x_n)} {x_1\in X_1,\ldots, x_n\in X_n}\text.
\]
If $s$ is a linear term, which means that each variable occurs in $s$ at most once, then it can be shown that
\[
s(X_1, \dots ,X_{n}) = \set{s(x_1,\dots, x_n) : x_1 \in X_1,\dots, x_n \in X_n}
\]
holds for arbitrary $X_i\in \mintc$ (but this does not hold for arbitrary terms in general).
This implies, as proved in 
\hideit{ M.\,N.\ Bleicher,  H.~Schneider and R.\,L.~Wilson~}\cite
{bleicher} and \hideit{W.~Taylor~}\cite{taylor},
that if a variety is defined by linear identities,
then it contains the complex algebra of each of its members.

Next, let $\brick$ denote the set 
$\set{X\subseteq A:  X^2\subseteq T \text{ and } X\neq \varnothing}$. Since it is clearly a subalgebra of the complex algebra of $\alg A$, the paragraph above implies that $\alg \brick=(\brick,F)$ belongs to $\var V$. 
Let $B=\setm{(x,Y)\in A\times \brick }{x\in Y}$. Then $\alg B=(B,F)$ also belongs to $\var V$ since it is a subalgebra of $\alg A\times \alg \brick$. Define $\bgq=\setbm{\bigl((x_1,Y_1),(x_2,Y_2)\bigr) \in {B}^2}{Y_1=Y_2  }$. As 
the kernel of the second projection from $\alg B$ to $\alg \brick$, 
it is a congruence of {$\alg B$}. The first projection $\gf \colon \alg B\to \alg A$, $(x,Y)\mapsto x$, is a surjective homomorphism since, for every $x\in A$,
$x=\gf\bigl(x,\set x \bigr)$.

Clearly, if $\bigl((x_1,Y_1),(x_2,Y_2)\bigr)\in\bgq$, then 
$\set{x_1,x_2}\subseteq Y_1=Y_2\in\brick $
implies that $\bigl(\gf(x_1,Y_1),\gf(x_2,Y_2)\bigr)= (x_1,x_2)\in T$. Hence $\gf(\bgq)\subseteq T$. 
Conversely, let $(x_1,x_2)\in T$. Then, 
with $Y=\set{x_1,x_2}$, we have that  $(x_1,Y), (x_2,Y)\in B$,  $\bigl((x_1,Y), (x_2,Y)\bigr)\in\bgq$, and
$x_i= \gf(x_i,Y)$. This implies that  
$(x_1,x_2)\in\gf(\bgq)$, and we conclude that $T\subseteq \gf(\bgq)$.
\end{proof}

\end{document}